\documentclass[oneside,11pt]{amsart}
\usepackage{amsmath, amsfonts,amsthm,times,graphics}

 \makeatletter
\renewcommand*\subjclass[2][2000]{%
  \def\@subjclass{#2}%
  \@ifundefined{subjclassname@#1}{%
    \ClassWarning{\@classname}{Unknown edition (#1) of Mathematics
      Subject Classification; using '1991'.}%
  }{%
    \@xp\let\@xp\subjclassname\csname subjclassname@#1\endcsname
  }%
}
 \makeatother

\newtheorem{theorem}{Theorem}[section]

\newtheorem{lemma}[theorem]{Lemma}
\newtheorem{corollary}[theorem]{Corollary}
\newtheorem{proposition}[theorem]{Proposition}
\theoremstyle{definition}
\newtheorem{definition}[theorem]{Definition}
\newtheorem{conjecture}[theorem]{Conjecture}
\newtheorem{example}[theorem]{Example}

\newcommand{\Mod}{\mathrm{Mod}\,}

\newtheorem{remark}[theorem]{Remark}
\numberwithin{equation}{section}
\newcounter{minutes}
\divide\time by 60
\newcounter{hours}
\multiply\time by 60 \addtocounter{minutes}{-\time}
\begin{document}

%On some sharp inequalities related to Dirichlet's integral

\title[Deformations of Annuli with Smallest Mean Distortion]{Deformations of Annuli on Riemann surfaces with Smallest Mean Distortion}

\author{David Kalaj}
\address{University of Montenegro, Faculty of Natural Sciences and
Mathematics, Cetinjski put b.b. 8100 Podgorica, Montenegro}
\email{davidk@t-com.me} % \maketitle

\maketitle

\begin{abstract}
Let $A$ and $A'$ be two circular annuli and let $\rho$ be a radial
metric defined in the annulus $A'$. Consider the class $\mathcal
H_\rho$ of $\rho-$harmonic mappings between $A$ and $A'$. It is
proved recently by Iwaniec, Kovalev and Onninen that, if $\rho=1$
(i.e. if $\rho$ is Euclidean metric) then $\mathcal H_\rho$ is not
empty if and only if there holds the Nitsche condition (and thus is
proved the J. C. C. Nitsche conjecture). In this paper we formulate
a condition (which we call $\rho-$Nitsche conjecture) with
corresponds to $\mathcal H_\rho$ and define $\rho-$Nitsche harmonic
maps. We determine the extremal mappings with smallest mean
distortion for mappings of annuli w.r. to the metric $\rho$. As a
corollary, we find that $\rho-$Nitsche harmonic maps are Dirichlet
minimizers among all homeomorphisms $h:A\to A'$. However, outside
the $\rho$-Nitsche condition of the modulus of the annuli, within
the class of
homeomorphisms, no such energy minimizers exist.  % However,
%outside the $\rho-$Nitsche range of the modulus of the annuli,
%within the class of homeomorphisms, no such energy minimizers exist.
This extends some recent results of Astala, Iwaniec and Martin
(ARMA, 2010) where it is considered the case $\rho=1$ and
$\rho=1/|z|$.
\end{abstract}
%\tableofcontents
\section{Introduction}
\subsection{Mappings
of finite distortion} A homeomorphism $w = f (z)$ between planar
domains ƒ$\Omega$ and ƒ$D$ has finite distortion if

a) $f$ lies in the Sobolev space $W^{1,1}_{loc} (\Omega,ƒD)$ of
functions whose first derivatives are locally integrable, and

b) $f$ satisfies the distortion inequality $$| f_{\bar z} | \le \mu
(z)| f_z |,$$ $0\le \mu(z) < 1$ almost everywhere in ƒ$\Omega$. Such
mappings are generalizations of quasiconformal homeomorphisms where
one works with the stronger assumption $\mu(z)\le k < 1$. Mappings
of finite distortion have found considerable interest in geometric
function theory and the mathematical theory of elasticity. A
comprehensive overview of the theory of mappings of finite
distortion in two-dimensions can be found in \cite{ar1}. The
Jacobian determinant of a mapping $f$ of finite distortion is
non-negative almost everywhere, since $$J_f(z)=|f_z|^2-|f_{\bar
z}|^2=(1-\mu(z)^2)|f_z|^2\ge 0.$$

The distortion function of particular interest to us in this article
is defined by the rule \begin{equation}\label{disor}K(z, f ) =
\frac{| f_z |^2 + | f_{\bar z} |^2 |}{ |f_z |^2 - | f_{\bar z} |^2}
=\frac{\|Df(z)\|^2}{J_ f( z )}\end{equation} if $J_f (z)
> 0.$ Here $$\|A\|^2
=\frac{1}{2}\mathrm{Tr}(A^TA)$$ is the square of mean
Hilbert-Schmidt norm. We conveniently set $K(z, f ) = 1$ if $f_z =
f_{\bar z} = 0$. Notice that then $K(z, f )=1$ and we have the
equality $K(z, f ) = 1$ if and only if $f$ is conformal, by the
Looman Menchoff theorem.
\subsection{Radial metrics}

By $\Bbb U$ we denote the unit disk $\{z:|z| <1\}$, by
$\overline{\Bbb C}$ is denoted the extended complex plane.  Let
$\Sigma$ be a Riemannian surface over a domain $C$ of the complex
plane or over $\overline{\Bbb C}$ and let $p: C\mapsto \Sigma$ be a
universal covering. Let $\rho_\Sigma$ be a conformal metric defined
in the universal covering domain $C$ or in some chart $D$ of
$\Sigma$. It is well-known that $C$ can be one of three sets: $\Bbb
U$, $\Bbb C$ and $\overline{\Bbb C}$. Then the distance function is
defined by

$$d(a,b) = \inf_{a,b\in
\gamma}\int_0^1\rho_\Sigma(\tilde\gamma(t))|\tilde\gamma'(t)|dt,$$
where $\tilde \gamma$, $\tilde\gamma(0)=0$, is a lift of $\gamma$,
i.e. $p(\tilde\gamma(t)) = \gamma(t)$, $\gamma(0) = a$, $\gamma(1) =
b$.

The Gauss curvature of the surface (and of the metric $\rho_\Sigma$)
is given by

$$K = -\frac{\Delta \log \rho_\Sigma}{\rho_\Sigma^2}.$$

In this paper we will consider those surfaces $\Sigma$, whose metric
have the form

$$\rho_\Sigma(z) = h(|z|^2),$$ defined in some chart $A'=\{z:\tau< |z|<\sigma\} $ of $\Sigma$
(not necessarily in the whole universal covering surface). Here $h$
is an positive twice differentiable function. We call these metrics
\emph{radial symmetric}. \begin{definition}\label{define} The radial
metric $\rho$ is called a \emph{regular metric} if $$\inf_{\tau< s<
\sigma}s\rho(s)=\lim_{s\to \tau+0}s\rho(s)$$ and has bounded Gauss
curvature $K$.
\end{definition}

Euclidean metric $\rho(z)=1$, and the metric $\rho(z)=\frac{1}{|z|}$
are regular metrics which are considered by Astala, Iwaniec and
Martin in the paper \cite{ar}. The authors settled corresponding
problems of deformations of annuli with smallest mean distortion
with respect to these two metrics. The aim of this paper is to
extend the results in \cite{ar} for all regular metrics (see
Theorem~\ref{or}, Theorem~\ref{te}, Corollary~\ref{co} and
Corollary~\ref{ro}).

In { Appendix} below it is showed that, most of known metrics with
constant Gauss curvatures defined in annuli are regular metrics.
Such metrics are for example, hyperbolic metric
$$\lambda(z) = \frac{2}{1-|z|^2},\;\;\; K=-1$$ defined in the unit disk $\mathbb U$ and Riemann metric $$\eta(z) =
\frac{2}{1+|z|^2},\;\;\; K=1$$ defined in the Riemann sphere
$S^2:=\overline{ \mathbb C}$.

\subsection{Harmonic mappings between Riemann surfaces} Let $(M,\sigma)$ and $(N,\rho)$ be
Riemann surfaces with metrics $\sigma$ and $\rho$, respectively. If
a mapping $f:(M,\sigma)\to(N,\rho)$ is $C^2$, then $f$ is said to be
harmonic (to avoid the confusion we will sometimes
say$\rho$-harmonic) if

\begin{equation}\label{el}
f_{z\overline z}+{(\log \rho^2)}_w\circ f f_z\,f_{\bar z}=0,
\end{equation}
where $z$ and $w$ are the local parameters on $M$ and $N$
respectively. Also $f$ satisfies \eqref{el} if and only if its Hopf
differential
\begin{equation}\label{anal}
\Psi=\rho^2 \circ f f_z\overline{f_{\bar z}}
\end{equation} is a
holomorphic quadratic differential on $M$.

For $g:M \mapsto N$ the energy integral is defined by

\begin{equation}\label{harel} E_\rho[g]=\int_{M}
(|\partial g|^2+|\bar \partial g|^2) dV_\sigma,
\end{equation}
% Now consider the functional $f\to E(f,\rho)$ and minimize it over all mappings $f$
% satisfying the condition $f|_{\partial \Omega}=g$
where $\partial g$, and $\bar \partial g$ are the partial
derivatives taken with respect to the metrics $\varrho$ and
$\sigma$, and $dV_\sigma$ is the volume element on $(M,\sigma)$.
Assume that energy integral of $f$ is bounded. Then $f$ is harmonic
if and only if $f$ is a critical point of the corresponding
functional where the homotopy class of $f$ is the range of this
functional. For this definition and the basic properties of harmonic
maps see \cite{sy}.

 It follows from the definition that,
\begin{lemma}\label{poli} If $a$ is holomorphic and $f$
is harmonic, then $f\circ a$ is harmonic.\end{lemma}

Using the fact that the function defined in \eqref{anal} is
holomorphic, the following well known lemma can be proved (see e.g.
\cite{km}).

\begin{lemma}\label{le} Let $(S_1, \rho_1)$ and $(S_2,\rho_2)$ and
$(R,\rho)$ be three Riemann surfaces. Let $g$ be an isometric
transformation of the surface $S_1$ onto the surface $S_2$:
$$\rho^2_1(\omega)|d\omega|^2=\rho^2_2(w)|dw|^2, \, \, w=g(\omega).$$
Then $f:R\mapsto S_1$ is $\rho_1$- harmonic if and only if $g\circ
f: R \mapsto S_2$ is $\rho_2$- harmonic. In particular, if $g$ is an
isometric self-mapping of $S_1$, then $f$ is $\rho_1$- harmonic if
and only if $g\circ f$ is $\rho_1$- harmonic.
\end{lemma}

\begin{example}
Let  $\rho$ be the Riemann metric $$\rho=\frac{2}{1+|z|^2}.$$
Equation \eqref{el} becomes
\begin{equation}\label{rimequ}u_{z\bar
z}-\frac{2\bar u}{1+|u|^2}u_z\cdot u_{\bar z}=0.\end{equation}

Notice this important example. The Gauss map of a surface $\Sigma$
in $\Bbb R^3$ sends a point on the surface to the corresponding unit
normal vector $\mathbf{n}\in \overline{\Bbb C} \cong S^2$. In terms
of a conformal coordinate $z$ on the surface, if the surface has
{\it constant mean curvature}, its Gauss map $\mathbf{n}: \Sigma
\mapsto \overline{\Bbb C}$, is a Riemann harmonic map \cite{rv}.
\end{example}

\begin{example}
If $u:\Bbb U\mapsto \Bbb U$ is a harmonic mapping with respect to
the hyperbolic metric $$\lambda=\dfrac{2}{1-|z|^2}$$ then
Euler-Lagrange equation of $u$ is
\begin{equation}\label{eqpo}
u_{z\bar z}+\frac{2\bar u}{1-|u|^2}u_z\cdot u_{\bar z}=0.
\end{equation}

An important example of hyperbolic harmonic mapping is the Gauss map
of a space-like surfaces with constant mean curvature $H$ in the
Minkowski $3$-space $M^{2,1}$ (see \cite{ctr}, \cite{mt} and
\cite{wan}).

\end{example}

\section{Radial $\rho-$harmonic mappings and $\rho-$Nitsche conjecture}
The conjecture in question concerns the existence of a harmonic
homeomorphism between circular annuli $A(r,1)$ and $A(\tau,\sigma)$,
and is motivated in part by the existence problem for
doubly-connected minimal surfaces with prescribed boundary. In 1962
Nitsche \cite{n} observed that the image annulus cannot be too thin,
but it can be arbitrarily thick (even a punctured disk). Then he
conjectured that for such a mapping to exist we must have the
following inequality, now known as the Nitsche bound:
$$\frac{\sigma}{\tau}\ge \frac 12\left(\frac 1r+r\right).$$ For some results concerning the partial solution of Nitsche conjecture see
papers \cite{israel}, \cite{weit} and \cite{Al}. For the
generalization of this conjecture to $\Bbb R^n$ and some related
results we refer to \cite{jmaa}. For the case of hyperbolic harmonic
mappings we refer to \cite{h}. Some other generalization has been
done in \cite{israel} (see Proposition~\ref{fikin} below). The
Nitsche conjecture for Euclidean harmonic mappings is settled
recently in \cite{conj} by Iwaniec, Kovalev and Onninen, showing
that, only radial harmonic mappings
$$h(\zeta) = C\left(\zeta-\frac{\omega}{\overline{\zeta}}\right),$$
$C\in\mathbb C$, $\omega\in \mathbb R$, $|C|(1-\omega)=\sigma$,
which inspired the Nitsche conjecture, make the extremal distortion
of rounded annuli.

In this section, we will state a similar conjecture with respect to
$\rho-$ harmonic mappings. In order to this, we will find examples
of radial $\rho$-harmonic maps between annuli. We put
$$w(z) = g(s) e^{it},\ \ z=se^{it}$$ where $g$ is a increasing or a decreasing function to
be chosen. This will include all radial harmonic mappings.

Direct calculations yield
\begin{equation}\label{prima}w_{z\bar z} = \frac 14 \Delta w = \frac{1}{4s^2}\left(s^2
w_{ss} +s w_s + w_{t t}\right) \end{equation} and
\begin{equation}\label{seconda}w_z w_{\bar z} =
\frac{1}{4s^2} (s^2w_s^2 - w_t^2).\end{equation} Inserting this into
harmonic equation \eqref{el}, we obtain
$$s^2 g'' + sg' - g +\frac{2\rho_w}{\rho} (s^{2}{g'}^2 - g^2) = 0.$$
Let $s = e^x$ and
\begin{equation}\label{varo}\varrho = \dfrac{1}{\rho}.\end{equation} Put $y(x) = g(e^x)$. Then the corresponding
differential equation is
$$y''-y = \frac{\varrho'(y)}{\varrho(y)}(y'^2-y^2).$$ After some
changes we obtain that, the general solution to this equation is
$$x+c_1 = \int \frac{dy}{\sqrt{y^2+c\varrho^2 }},$$ where $c$ and
$c_1$ are certain constants. The mapping $w$ given by
\begin{equation}\label{w}w(se^{it}) =
q^{-1}(s)e^{it},\end{equation} where
\begin{equation}\label{var}q(s) =\exp(\varphi(s))=\exp\left(\int_{\sigma}^s
\frac{dy}{\sqrt{y^2+c\varrho^2 }}\right), \ \tau\le s\le \sigma,
\end{equation} and $c$ satisfies the condition:
\begin{equation}\label{unt} y^2+c\varrho^2(y)\ge 0, \;\text{for} \; \tau\le s\le \sigma
,\end{equation} is a $\rho$-harmonic mapping between annuli
$A=A(r,1)$ and $A'=A(\tau,\sigma)$, where
\begin{equation}\label{rr}r=\exp\left(\int_{\sigma}^\tau \frac{dy}{\sqrt{y^2+c\varrho^2
}}\right).\end{equation} The harmonic mapping $w$ is normalized by
$$w( e^{it})=\sigma e^{it}.$$ The mapping $w=h^c(z)$ is a diffeomorphism, and we will call it \emph{$\rho$-Nitsche mapp}. From now on we will assume that the
metric $\rho$ is regular in the sense of Definition~\ref{define}.
Then \eqref{unt} is equivalent to
\begin{equation}\label{unti}\tau^2+c\varrho^2(\tau)\ge
0.\end{equation} Accordingly, for $c=-\tau^2\rho^2(\tau)$, we have
well defined function $$q^{\#}(s)=\exp\left(\int_{\sigma}^s
\frac{dy}{\sqrt{y^2-\tau^2\rho^2(\tau)\varrho^2 }}\right), \tau\le
s\le \sigma.$$ The mapping $h^{\#}:A\to A'$ defined by
$h^{\#}(se^{it})=(q^{\#})^{-1}(s)e^{it}$ is called the
\emph{critical Nitsche map}.

For $\tau\le s\le \sigma\le 1$ we have:
\begin{equation}\label{case}s^2\varphi'(s)^2-1  =\frac{-c}{s^2\rho^2+c}\left\{\begin{array}{ll}
                                                      \le 0, & \hbox{if $c\ge 0$;} \\
                                                      \ge 0, & \hbox{if $-\tau^2\rho^2(\tau)\le c\le 0$.}
                                                    \end{array}
                                                  \right.\end{equation}
Notice that, the mapping $$f^{c}(se^{it})=q(s)e^{it}:A\to A'$$ is
the inverse of the harmonic diffeomorphism $w$.
\begin{conjecture} {\it Let $\rho$ be a regular metric. % such that
%$\displaystyle\sup_{\tau< s< \sigma}\rho(s)<\infty.$
 If $r<1$, and there exists a $\rho-$ harmonic mapping of the annulus
$A'=A(r,1)$ onto the annulus $A=A(\tau,\sigma)$, then
\begin{equation}\label{nit}r\ge \exp\left(\int_{\sigma}^{\tau}
\frac{\rho(y)dy}{\sqrt{y^2\rho^2(y)-\tau^2\rho^2(\tau)
}}\right).\end{equation}}
\end{conjecture}
{\it Notice that if $\rho =1$, then this conjecture coincides with
standard Nitsche conjecture.}

The following example assert that, in the settings of the previous
conjecture a upper bound for $r$ do not holds. In other words the
image domain it can be arbitrarily thick (even a punctured disk).
This differs harmonic mappings from conformal mappings and
quasiconformal mappings.
\begin{example}
Let $\rho$ be a metric defined on the unit disk $\mathbb U$, and
take $A'=A(0,1)$ ($0=\tau<\sigma =1$). Let $c>0$ and
$$q(s)=
\exp\left(\int_1^s\frac{\rho(s)ds}{\sqrt{c+s^2\rho^2(s)}}\right).$$
Define $w(z) =q^{-1}(s)e^{it}$. Then $w$ is a $\rho-$harmonic
diffeomorphism between annuli $A(r_c,1)$ and the degenerated annuli
$A(0,1)$. Here
$$0<r_c=
\exp\left(-\int_0^1\frac{\rho(s)ds}{\sqrt{c+s^2\rho^2(s)}}\right)<1$$
and $$\lim_{c\to+\infty} r_c=1.$$
\end{example}
We call the inequality \eqref{nit} $\rho-${\it Nitsche condition}.
The converse inequality \begin{equation}\label{notnit}r<
\exp\left(\int_{\sigma}^{\tau}
\frac{\rho(y)dy}{\sqrt{y^2\rho^2(y)-\tau^2\rho^2(\tau)
}}\right),\end{equation} we will call {\it the fatness condition}.

If $r<1$ satisfies the condition \eqref{nit}, then by continuity
argument, there exists an $c$ satisfying
$$c\ge -\tau^2\rho^2(\tau)$$ such that
\begin{equation}\label{cnit}r= \exp\left(\int_{\sigma}^{\tau}
\frac{\rho(y)dy}{\sqrt{y^2\rho^2(y)+c}}\right).\end{equation} The
following theorem is a partial result in solving the previous
conjecture
\begin{proposition}\cite{kalaj}\label{fikin}
Assume that $\rho$ is a metric defined in the disk $\{z:|z|\le
\sigma\}$ with positive or negative Gauss curvature $K(z)$ and let
$$h(x)=\int_{0}^x\rho(t)dt, \ \ 0\le x\le \sigma .$$ If there exists a $\rho-$harmonic diffeomorphism
between the annuli $A=A(r,1)=\{ z\in \Bbb C: r<|z|< 1\}$  and $A'=\{
z\in \Bbb C: \tau<|z|< \sigma \}$, then
\begin{equation}\label{desired}
\frac{h(\sigma)}{h(\tau)}\geq 1+\frac{\tau}{2h(\tau)}\log^2
{r}\left\{
                             \begin{array}{ll}
                                \left(t\rho(t)\right)'|_{t=\tau}, & \hbox{if $K$ is negative;} \\
                               \left(t\rho(t)\right)'|_{t=\sigma}, & \hbox{if $K$ is positive.}
                             \end{array}
                           \right.
\end{equation}
\end{proposition}
\section{Statement of the main results}

The classical formulations of the extremal Gr\"otsch and
Teichm\"uller problems are concerned with finding mappings
ƒ$\Omega\to D$ in some class (for instance, with free or prescribed
boundary values) which have smallest $L^\infty$-norm of the
distortion function, thus "extremal quasiconformal mappings". In
this article we shall investigate mappings in some class which
minimize integral means with respect to appropriate metrics  of the
distortion function $\mathbb K(z, f )$. The case of bounded simply
connected domains, without boundary data, is trivial; the extremals
are the conformal mappings of ƒ$\Omega$ onto ƒ$D$ asserted to exist
by the Riemann mapping theorem (the simply connected case where the
boundary data is prescribed is solved in \cite{ar2}). For the free
boundary problem, Astala, Iwaniec and Martin \cite{ar} considered
the first nontrivial case where there are conformal invariants;
namely doubly connected domains and, in particular, annuli. Given
two annuli $$ A'=\{w: \tau<|w|<\sigma\}, \;\;  A = \{z : r < |z| <
1\},$$ they consider homeomorphisms of finite distortion $f : A' \to
A$ with respect to the Euclidean metric and the metric $\rho =
1/|z|$. We shall consider the same problem but for an arbitrary
\emph{regular metric}. Here note that $| f |$ extends continuously
to $\overline{A'}$, with values $r$ and $1$ on the boundary of $A$.
We shall normalize our mappings in the obvious way so that $$| f
(z)| = r\; \text{for}\; |z| = \tau\; \text{and}\; | f (z)| =
\sigma\; \text{for}\; |z| = 1$$ Let $\mathcal F = \mathcal F(A',A )$
denote the family of all normalized homeomorphisms $f : A' \to A$ of
finite distortion. Since $A'$ and $A$ are certainly diffeomorphic
$\mathcal F\neq \emptyset$.

Let $z=x+iy=se^{it}$ and
$$dm(z) = dxdy=s ds dt$$ be the usual \emph{Lebesgue measure} on the complex plane $\Bbb C$.  The integral mean of the distortion function $K(z,
f )$ which concern us in this work is
$$\mathcal K_\rho[f]=\int_{A'}\mathbb K(z,f)\rho^2(z)dm(z),$$ where $\rho$ is a given \emph{regular metric}.
The minimization problem we address here is to evaluate the
following infimum
\begin{equation}\inf \{\mathcal K_\rho[f]: f\in\mathcal
F(A',A)\}.\end{equation}

Further, we should decide if the infimum is attained and, in that
case, prove uniqueness (up to the obvious rotational symmetry of the
annuli). The concept of the conformal modulus will prove useful in
proving our results.  It is convenient to take the following
definition of the modulus of an annulus $A(p,q):=\{z:p<|z|<q\}$
\begin{equation}\label{mod}\Mod(A(p,q)) = 2\pi \log \frac q p =
\int_{A(p,q)} \frac{dm(z)}{ |z|^2}.\end{equation} Note that, the
standard definition of modulus is indeed
$\mathrm{mod}(A(p,q))=\frac{1}{2\pi}\log \frac qp$. Every
topological annulus $R$ is conformally equivalent to a round annulus
$A$, and we can set $\Mod(R) = \Mod(A)$.

The main theorems of this paper are the following:

\begin{theorem}\label{or} Let $\rho$ be a regular metric. Let $A$ and $A'$ be annuli satisfying the condition
\eqref{nit}. Among all mappings $f\in \mathcal F(A, A')$ the infimum
of \begin{equation}\label{distortion}\int_{A'} \mathbb K(z,f)
\rho^2(z) dm(z)\end{equation} is attained by the function
\begin{equation}\label{cf}f^{c}(z)=\exp\left(i(t+\alpha)+\int_{\sigma}^s
\frac{\rho(y)dy}{\sqrt{y^2\rho^2(y)+c }}\right), \ \ z= se^{it},
\alpha\in[0,2\pi),\end{equation} where $c$ is given by \eqref{cnit}.
Its inverse $h^{c}$ is $\rho-$harmonic between annuli $A'$ and $A$.

\end{theorem}

\begin{remark}
In the special cases where $\rho(z)= 1$, we easily obtain that
$$f^{c}(z)=C(c)\frac{z}{|z|}\left(|z|+\sqrt{|z|^2+c}\right)$$  which is the inverse of the Nitsche's
map $$h^c(\zeta) =
C'(c)\left(\zeta-\frac{\omega(c)}{\overline{\zeta}}\right),$$ where
$\omega(c)$ is a positive constant and $C(c),C'(c)\in \mathbb C$.

If $\rho(z) = |z|^{-1}$, and $z=se^{it}$, then $f^c$ is a power
function
$$f^c(z) = \sigma s^{-\alpha}|z|^{\alpha-1} z,\ \ \ \text{where}\ \
\alpha =\alpha(c)= \frac{\Mod(A')}{\Mod(A)}.$$ See
\cite[Theorem~1\&~Theorem~2]{ar} for the same conclusion.

\end{remark}

\begin{theorem}\label{te} Let $\rho$ be a regular metric.
Under the fatness condition \eqref{notnit}, the infimum at
\eqref{distortion} is not attained by any homeomorphism $f\in
\mathcal F(A,A')$ (more generally, it is not attained by any
continuous mapping of finite distortion of $A$ onto $A'$). Moreover
for the inverse of critical Nitsche map
$f^{\#}(se^{it})=e^{\varphi^{\#}(s)+it}$, where
\begin{equation}\label{wo}\varphi^{\#}(s) =\int_{\sigma}^s
\frac{dy}{\sqrt{y^2-\tau^2\rho^2(\tau)\varrho^2(y) }}, \ \tau\le
s\le \sigma, \end{equation} there holds the sharp inequality
\begin{equation}\label{after}\int_{A'}\mathbb K(z,f)\rho^2(s) dm(z)\ge \int_{A'}\mathbb
K(z,f^{\#})\rho^2(s) dm(z) +\frac{\tau^2\rho^2(\tau)}{2}\Mod
A(r,r'),\end{equation} where
$$r'=\exp\left(\int_{\sigma}^{\tau}
\frac{\rho(y)dy}{\sqrt{y^2\rho^2(y)-\tau^2\rho^2(\tau) }}\right)\;
\; (r<r'<1).$$
\end{theorem}
The minimization of the integral means of the distortion functions
of homeomorphisms $f: A'\to A$ turns to be equivalent to the
Dirichlet type problem for the inverse mapping $h= f^{-1}: A\to A'$.
If a homeomorphism $f\in W_{loc}^{1,1}(A',A)$ has integrable
distortion, then $h\in W^{1,2}(A, A')$ and we can consider the
energy functional \begin{equation}\label{po}E_\rho[h]=\int_{A}
\|Dh(\zeta)\|^2\rho^2(h(\zeta)) dm(\zeta)= \int_{A'} \mathbb
K(z,f)\rho^2(z)dm(z).\end{equation}

In general, the converse is not true, because the inverse of a
homeomorphism $h\in W^{1,2}(A,A')$ need not belong to the Sobolev
class $W_{loc}^{1,1}(A,A')$. It has bounded variation but fails to
be ACL (absolutely continuous on lines), see \cite{heko} for related
results. As in \cite{ar} we prove the correction lemma and overcame
this problem.

Accordingly, for every homeomorphism $h\in W^{1,2}(A,A')$, we can
construct a homeomorphism $\tilde h\in W^{1,2}(A,A')$, with
$E_\rho[\tilde h]\le E_\rho[h]$, whose inverse lies in $\mathcal
F(A',A)$. As a consequence, the minimization problem for $\mathcal
K_{\rho}[f]$ is equivalent to the minimization problem for
$E_\rho[h]$.

The energy integral defined in \eqref{harel} coincides with the
energy integral in \eqref{po}. We should point out that, the image
surface $M$ is indeed the annulus $A$ with metric $\rho$ and make
use of the formulas
$$|\partial g|^2=\frac{\rho^2\circ g|g_z|^2}{\sigma^2(z)},\,\, |\bar\partial g|^2=\frac{\rho^2\circ g|g_{\overline z}|^2}{\sigma^2(z)} \text{ and }
dV_\sigma = \sigma^2(z)dm(z).$$ From Theorem~\ref{or} and
Lemma~\ref{corect} we deduce
\begin{corollary}\label{ro} Let $\rho$ be a regular metric.
Within the Nitsche rang \eqref{nit}, for the annuli $A$ and $A'$,
the absolute minimum of the energy integral $$h\to E_\rho[h],\ \
h\in W^{1,2}(A,A')$$ is attained by a $\rho-$Nitsche map
$$h^{c}(z)=q^{-1}(s)e^{i(t+\beta)}, \; z=se^{it}, \
\beta\in[0,2\pi),$$ where $$q(s) =\exp\left(\int_{\sigma}^s
\frac{dy}{\sqrt{y^2+c\varrho^2 }}\right), \tau<s<\sigma.$$
\end{corollary}
\begin{remark}
If $R$ is an doubly connected surface, conformal to a given annulus
$A$, and $(R',\rho')$ another annulus isometric to a given annulus
$(A',\rho)$, then by Lemma~\ref{poli} and Lemma~\ref{le}, the
minimization of the energy integral
$$h\to E_\rho[h],\ \ h\in W^{1,2}(A,A')$$ is equivalent to the
minimization of the energy integral
$$k\to E_{\rho'}[k],\ \ k\in W^{1,2}(R,R').$$ Thus
Corollary~\ref{ro} can be formulated in terms of not-necessarily
rounded annuli.
\end{remark}

From Theorem~\ref{te} and Lemma~\ref{corect} we deduce
\begin{corollary}\label{co} Let $\rho$ be a regular metric.
Outside the $\rho-$Nitsche range \eqref{nit} for the annuli $A$ and
$A'$, the infimum of the energy functional $E_\rho[h]$ is not
attained by any homeomorphism $h'\in W^{1,2}(A,A')$.
\end{corollary}

\begin{remark}
Theorem~\ref{or}, Theorem~\ref{te}, Corollary~\ref{co} and
Corollary~\eqref{ro} are generalizations of corresponding
\cite[Thm~1,~Thm~2,~Thm~3,~Cor~1,~Cor~2~ and~ Cor~3]{ar}.
\end{remark}
\section{Proof of Theorem~\ref{or} and Theorem~\ref{te}}
We need the following elementary formulas in the sequel. Let $z=s
e^{it}$. Then
\begin{equation}\label{f}|f_z|^2+|f_{\bar z}|^2 = \frac12(|f_s|^2+s^{-2}|f_t|^2)\end{equation} and

\begin{equation}\label{j}J_f(z) = \frac{1}{s}\Im(f_t
\overline{f_s}).\end{equation}

From \eqref{f} and \eqref{j} we obtain

\begin{equation}\label{k}\mathbb K(z,f) =
\frac{s|f_s|^2+s^{-1}|f_t|^2}{2\Im(f_t\overline{f_s})}.\end{equation}

If a mapping $f$ is radial stretching between annuli $A'$ and $A$,
then for some increasing function $P(s), \tau<s<\sigma $ there holds
the formula $$f(se^{it}) = P(s)e^{it}.$$ If $\Phi(s)=\log P(s)$,
then we can express the distortion function as
\begin{equation}\label{iva}
\mathbb K(z,f) = \frac 12\left(s\Phi'(s)+\frac{1}{s\Phi'(s)}\right).
\end{equation}
%See \cite[Eq.~17]{ar} for similar formula.

\begin{lemma}\label{lepo}
Let $f$ be a mapping of finite distortion and $\varphi$ be a
differentiable monotonic function. For $z\in A^+:=\{z: J_f(z)>0\}$
we have the following equivalent inequalities

\begin{equation}\label{ine}\mathbb K(z,f)\ge s \varphi'(s) +
\frac{1-s^2(\varphi'(s))^2}{2s^2J_f(z)}|f_t|^2\end{equation} and

\begin{equation}\label{eni}\mathbb K(z,f)\ge \frac{1}{s\varphi'(s)} +
\frac{s^2\varphi'(s)^2-1}{\varphi'(s)^2}\frac
{1}{2s^2J_f(z)}|f_s|^2.\end{equation} In both inequalities the
equality is attained a.e. if and only if $f_s=-i \varphi'(s)f_t$ for
$z\in A^+$.

\end{lemma}

\begin{proof}
In view of \eqref{k}, it is easily to verify that, the following
trivial inequality
\begin{equation}\label{ineni}\left|f_t - \frac{i f_s}{\varphi'(s)}\right|\ge
0\end{equation} is equivalent with both inequalities \eqref{ine} and
\eqref{eni}.
\end{proof}

\subsection{Proof of inequalities}
\begin{lemma}[The main lemma]\label{ndihma}
Let $\rho$ be a regular metric and let $f:A'\to A$ be a homomorphism
of finite distortion between annuli $A'=A(\tau,\sigma)$ and
$A=A(r,1)$.

a) Assume that there holds the $\rho-$Nitsche condition \eqref{nit}
and let $c$ be defined by \eqref{cnit}. For
$f^{c}(se^{it})=e^{\varphi(s)+it}$, where
$$\varphi(s) =\int_{\sigma}^s \frac{dy}{\sqrt{y^2+c\varrho^2(y) }}, \
\tau\le s\le \sigma, $$ there holds the inequality
$$\int_{A'}\mathbb K(z,f)\rho^2(s) dm(z)\ge \int_{A'}\mathbb
K(z,f^{c})\rho^2(s) dm(z).$$

b) Assume that there holds the fatness condition \eqref{notnit}. Now
we make use of inverse of critical Nitsche map. For
$f^{\#}(se^{it})=e^{\varphi^{\#}(s)+it}$, where
$$\varphi^{\#}(s) =\int_{\sigma}^s \frac{dy}{\sqrt{y^2-\tau^2\rho^2(\tau)\varrho^2(y) }}, \
\tau\le s\le \sigma, $$ there holds the inequality
$$\int_{A'}\mathbb K(z,f)\rho^2(s) dm(z)\ge \int_{A'}\mathbb
K(z,f^{\#})\rho^2(s) dm(z) +\frac{\tau^2\rho^2(\tau)}{2}\Mod
A(r,r'),$$ where
$$r'=\exp\left(\int_{\sigma}^{\tau}
\frac{\rho(y)dy}{\sqrt{y^2\rho^2(y)-\tau^2\rho^2(\tau) }}\right)\;
\; (r'>r).$$

\end{lemma}
\begin{proof} Let $z=se^{it}$, $s=|z|$, $t\in[0,2\pi)$ and $A^+=\{z: J_f(z)>0\}$. We will apply Lemma~\ref{lepo} with $$\varphi(s) =\int_{\sigma}^s \frac{dy}{\sqrt{y^2+c\varrho^2 }}, \
\tau\le s\le \sigma.$$
\\
\emph{Proof of a)}. Since $\varrho=1/\rho$, we have
$$1-s^2\varphi'(s)^2 =\frac{c}{s^2\rho^2+c}.$$ In view of this fact, we divide the proof
into two cases.
\\
{$\bullet$ \bf The case $c\ge 0$}. Observe first that, for almost
every $z=se^{it}\in A'$
\begin{equation}\label{gati}s\varphi'(s)=\frac{s}{\sqrt{s^2+c\varrho^2(s)}}\le
1\le \mathbb K(z,f).\end{equation} According to \eqref{ine} and
\eqref{gati}, we have
\begin{equation}\label{1}\begin{split}\int_{A'}\mathbb K(z,f)\rho^2(s) &\ge \int_{A'\setminus A^+} s\rho^2(s)
\varphi'(s) +\int_{A^+} s\rho^2(s) \varphi'(s) \\& + \int_{A^+}
\rho^2(s)\frac{1-s^2(\varphi'(s))^2}{2s^2J_f(z)}|f_t|^2 \\&=
\int_{A'} s\rho^2(s) \varphi'(s)  + \int_{A^+} \frac{c}{2s^2(s^2+c
\varrho^2(s))}\frac{|f_s|^2}{J_f(z)}. \end{split}\end{equation} By
H\"older inequality we obtain \begin{equation}\label{2}\int_{A^+}
\frac{|f_t|}{s|f|\sqrt{2(s^2+c \varrho^2(s))}} \le \left(\int_{A^+}
\frac{1}{2s^2(s^2+c
\varrho^2(s))}\frac{|f_t|^2}{J_f(z)}\right)^{1/2}\left(\int_{A^+}
\frac{J_f(z)}{|f|^2}\right)^{1/2}.\end{equation} By
\cite[Lemma~1]{ar} it follows that
\begin{equation}\label{l1}\int_{A^+} \frac{J_f(z)}{|f|^2}\le\int_{A'}
\frac{J_f(z)}{|f|^2}\le \Mod(A).\end{equation} Since
$$\int_{0}^{2\pi}\frac{|f_t(se^{it})|}{|f(se^{it})|}\ge
\left|\int_0^{2\pi}\frac{f_t(se^{it})}{f(se^{it})}dt\right|=2\pi,$$
from \eqref{l1}, we obtain
\[\begin{split}\int_{A'} \frac{1}{2s^2(s^2+c
\varrho^2(s))}\frac{|f_t|^2}{J_f(z)} &\ge
\frac{1}{\Mod(A)}\left(\int_{A'} \frac{|f_t|}{s|f|\sqrt{2(s^2+c
\varrho^2(s))}}\right)^2\\&\ge
\frac{4\pi^2}{\Mod(A)}\left(\int_{\tau}^\sigma
\frac{ds}{\sqrt{2(s^2+c \varrho^2(s))}}\right)^2.\end{split}\] On
the other hand, $$\int_{A'} s \rho^2(s) \varphi'(s) dm(z) =
2\pi\int_\tau^\sigma \frac{s^2\rho^2(s)}{\sqrt{s^2+c\varrho^2(s) }}
ds.$$ Therefore \[\begin{split}\int_{A'}\mathbb K(z,f)\rho^2(s)
dm(z)&\ge 2\pi\int_\tau^\sigma
\frac{s^2\rho^2(s)}{\sqrt{s^2+c\varrho^2(s) }} ds \\&+
\frac{4\pi^2}{\Mod(A)}\left(\int_\tau^\sigma
\frac{ds}{\sqrt{2(s^2/c+ \varrho^2(s))}}\right)^2.\end{split}\]
Since
\begin{equation}\label{modulus}\Mod(A) = 2\pi \log\frac 1r = 2\pi
\int_\tau^\sigma \frac{ds}{\sqrt{s^2+c\varrho^2(s)}},\end{equation}
it follows that
\[\begin{split}\int_{A'}\mathbb K(z,f)\rho^2(s) dm(z)&\ge 2\pi\int_\tau^\sigma
\frac{s^2\rho^2(s)}{\sqrt{s^2+c\varrho^2(s) }} ds+ \pi
c\int_\tau^\sigma \frac{ds}{\sqrt{s^2+
c\varrho^2(s))}}\\&=\pi\int_\tau^\sigma
\frac{s^2\rho^2(s)}{\sqrt{s^2+c\varrho^2(s) }} ds \\&+
\pi\int_\tau^\sigma \frac{s^2\rho^2(s)}{\sqrt{s^2+c\varrho^2(s) }}
ds+\pi c\int_\tau^\sigma \frac{ds}{\sqrt{s^2+ c\varrho^2(s))}}\\&
=\pi\int_\tau^\sigma \frac{s^2\rho^2(s)}{\sqrt{s^2+c\varrho^2(s) }}
ds+\pi \int_\tau^\sigma \rho^2(s)\sqrt{s^2+c\varrho^2(s)
}ds.\end{split}\] For $f^{c}(se^{it})=e^{\varphi(s)+it}$, where
$$\varphi(s) =\int_{\sigma}^s \frac{dy}{\sqrt{y^2+c\varrho^2 }}, \
\tau\le s\le \sigma, $$ by making use of formula \eqref{iva}, we
have
\[\begin{split}\int_{A'}\mathbb K(z,f^{c})\rho^2(s) dm(z)&=\frac{2\pi}{2} \int_\tau^\sigma
s\rho^2(s)\left({\varphi'(s)}{s}+\frac{1}{s\varphi'(s)}\right)ds\\&=\pi
\int_\tau^\sigma s\rho^2(s)\left(\frac{s}{\sqrt{s^2+c\varrho^2(s)
}}+\frac{\sqrt{s^2+c\varrho^2(s)}}{s}\right)ds.\end{split}
\]
Thus $$\int_{A'}\mathbb K(z,f)\rho^2(s) dm(z)\ge \int_{A'}\mathbb
K(z,f^{c})\rho^2(s) dm(z).$$
\\
{$\bullet$ \bf  The case $-\tau^2\rho^2(\tau)\le c\le 0$}. In this
case we make use of \eqref{eni}. Observe also that, for almost every
$z\in A'$
\begin{equation}\label{gatima}\frac{1}{s\varphi'(s)}=\frac{\sqrt{s^2+c\varrho^2(s)}}{s}\le
1\le \mathbb K(z,f).\end{equation} For $z\in A^+$, we obtain
$$\mathbb K(z,f)\rho^2(s)\ge
\frac{\rho^2(s)}{s\varphi'(s)}+\frac{-c\rho^2(s)}{2s^2\rho^2(s)}\frac{|f_s|^2}{J_f}.$$
Hence by \eqref{gatima}
\begin{equation}\label{spl}\begin{split}\int_{A'}\mathbb
K(z,f)\rho^2(s) &\ge \int_{A'\setminus
A^+}\frac{\rho^2(s)}{s\varphi'(s)}+
\int_{A^+}\frac{\rho^2(s)}{s\varphi'(s)}+\int_{A^+}\frac{-c}{2s^2}\frac{|f_s|^2}{J_f}\\&
=\int_{A'}\frac{\rho^2(s)}{s\varphi'(s)}+\int_{A^+}\frac{-c}{2s^2}\frac{|f_s|^2}{J_f}.\end{split}
\end{equation} By \cite[Lemma~3]{ar} and \eqref{modulus} we have
\begin{equation}\label{ge}\int_{A^+}\frac{1}{s^2}\frac{|f_s|^2}{J_f}\ge\Mod(A)=2\pi \int_\tau^\sigma
\frac{ds}{\sqrt{s^2+c\varrho^2(s)}}.\end{equation} On the other hand
\begin{equation}\label{farc}\int_{A'}\frac{\rho^2(s)}{s\varphi'(s)}=2\pi\int_{\tau}^\sigma \rho^2(s)\sqrt{s^2+c\varrho^2(s)
}ds.\end{equation} Let $$k(s):=\sqrt{s^2+c\varrho^2(s) }.$$ From
\eqref{spl}, \eqref{ge} and \eqref{farc} we obtain
\begin{equation}\begin{split}\int_{A'}\mathbb
K(z,f)\rho^2(s) &\ge -c\pi \int_\tau^\sigma \frac{ds}{k(s)}+
2\pi\int_{\tau}^\sigma \rho^2(s)k(s)ds\\&=\pi\int_{\tau}^\sigma
\rho^2(s)k(s)ds+\pi
\int_{\tau}^{\sigma}\frac{s^2\rho^2(s)}{k(s)}ds\\&-c\pi
\int_\tau^\sigma \frac{ds}{k(s)}+ \pi\int_{\tau}^\sigma
\rho^2(s)k(s)ds-\pi
\int_{\tau}^{\sigma}\frac{s^2\rho^2(s)}{k(s)}ds\\&=\int_{A'}\mathbb
K(z,f^{c})\rho^2(s) dm(z)+X,\end{split}\end{equation} where
\begin{equation*}\label{X}\begin{split}X &=\pi \int_\tau^\sigma \left(-\frac{c}{k(s)}+
\rho^2(s)k(s)-\frac{s^2\rho^2(s)}{k(s)}\right)ds=
0\end{split}.\end{equation*} This concludes the proof of  a).
\\
\emph{The proof of b)}.  We proceed similarly as in the proof of a).
For $\varphi = \varphi^{\#}$, according to \eqref{eni}, for $z\in
A^+$ we have
$$\mathbb K(z,f)\rho^2(s)\ge
\frac{\rho^2(s)}{s\varphi'(s)}+\frac{\tau^2\rho^2(\tau)}{2s^2}\frac{|f_s|^2}{J_f}.$$
In this case, for $c^{\#}=-\tau^2\rho^2(\tau)$, instead of
\eqref{ge} we have
\begin{equation}\label{arma}\int_{A^+}\frac{1}{s^2}\frac{|f_s|^2}{J_f}\ge\Mod(A)=2\pi\log\frac 1r>2\pi \int_\tau^\sigma
\frac{ds}{\sqrt{s^2+c^{\#}\varrho^2(s)}}.\end{equation} On the other
hand, the relation \eqref{farc} holds. Hence
\begin{equation}\begin{split}\int_{A'}\mathbb K(z,f)\rho^2(s)dm(z)\ge \int_{A'}\mathbb
K(z,f^{\#})\rho^2(s) dm(z)+Y,\end{split}\end{equation} where $$Y
=\frac{-c^{\#}}{2}\left(2\pi\log\frac 1r- 2\pi\log\frac
1{r'}\right)=\frac{\tau^2\rho^2(\tau)}{2}\Mod A(r,r').$$ This
completes the proof of the lemma.
\end{proof}
\begin{proof}[Proof of Theorem~\ref{or}]  In view of Lemma~\ref{ndihma}, a), it remains to prove
the equality statement. By \cite[Lemma~3]{ar}, the equality in
\eqref{ge} occurs if and only if almost everywhere on $A^+:=\{z\in
A': J_f(z)>0\}$ we have
\begin{equation}\label{re} \Re
\left(\frac{f_s}{f}\right)=\left|\frac{f_s}{f}\right|
\end{equation}
and
\begin{equation}\label{im} \Im
\left(\frac{f_t}{f}\right)=k
\end{equation}
for some constant $k>0$. On other hand, by Lemma~\ref{lepo} we have
the equality
\begin{equation}\label{pde}f_s=-i \varphi'(s)f_t,\end{equation}
for $z\in A^+$. Combining \eqref{re}, \eqref{im} and \eqref{pde} we
arrive at the following system of PDE's
\begin{equation}\label{equa}
\frac{f_t}{f}= ik \; \; \;\text{    and    }\; \; \; \frac{f_s}{f}=
k\varphi'(s).
\end{equation} Integrating the first equality over the unit circle
gives $k=1$. Let $g(s,t)=\log f(se^{it}).$ Then $g$ is well defined
function from  $[\tau,\sigma]\times [0,2\pi]$ onto $[\log
r,0]\times[0,2\pi]$. Then from \eqref{equa} we obtain
$$g=it+\phi(s)\;\; \text{ and }\; g=\varphi(s)+\psi(t),$$ for some functions $\phi$
and $\psi$. We obtain that, the general solution of system of PDE's
\eqref{equa} is
$$f(se^{it})=Ce^{\varphi(s)+it}.$$
\end{proof}
\begin{proof}[Proof of Theorem~\ref{te}]
The second part of Theorem~\ref{te} follows by Lemma~\ref{ndihma},
b). It remains  to show that, the inequality is sharp and it is not
attained by any homeomorphism. Otherwise, if $f\in \mathcal F(A',A)$
would satisfy \eqref{wo}, $f$ has to be of the form
$$f(se^{it})=C\exp\left(it+\int_{\sigma}^s \frac{dy}{\sqrt{y^2-\tau^2\rho^2(\tau)\varrho^2(y) }}\right), \
\tau\le s\le \sigma. $$ But then $r'=r$ which is a contradiction
with \eqref{notnit}. To show that the inequality is sharp we
construct the minimizing sequence.
\subsection{The minimizing sequence}
We are now given two round annuli $A$ and $A'$ with inner and outer
radii $r$, $1$ and $\tau$, $\sigma$ respectively. Moreover we assume
that the \emph{fatness condition} \eqref{notnit} is satisfied. To
show that the inequality \eqref{after} is sharp we construct a
family $f_n:A'\to A$ of mappings of finite distortion. Let
$$r'=\exp\left(\int_{\sigma}^{\tau}
\frac{\rho(y)dy}{\sqrt{y^2\rho^2(y)-\tau^2\rho^2(\tau) }}\right).$$
Then $r<r'<1$. Let $z=se^{it}$ and $n\in \mathbb N$ such that
$$1-r\left(\frac{\sigma}{\tau}\right)^n<0.$$ Define
$$f_n(z)=\left\{
                     \begin{array}{ll}
                      \exp\left(it+\displaystyle \int_{\sigma}^s \frac{dy}{\sqrt{y^2-\tau^2\rho^2(\tau)\varrho^2(y) }}\right)
 & \hbox{if }\ \ s_n\le s\le \sigma,
 \\
                       rz |z|^{n-1}/\tau^n & \hbox{if } \tau\le s\le s_n, %\\
                       %rz/\tau & \hbox{if } s\le \tau,
                     \end{array}
                   \right.$$ where $s_n$ is a solution of
the equation \begin{equation}\label{condi}
\exp\left(\int_{\sigma}^{s_n}
\frac{dy}{\sqrt{y^2-\tau^2\rho^2(\tau)\varrho^2(y) }}\right)
=r\left(\frac{s_n}{\tau}\right)^n.
\end{equation}
To show that $s_n:\tau<s_n<\sigma$ exists satisfying the condition
\eqref{condi} take $$p(s) =\exp\left(\int_{\sigma}^{s}
\frac{dy}{\sqrt{y^2-\tau^2\rho^2(\tau)\varrho^2(y) }}\right)
-r\left(\frac{s}{\tau}\right)^n.$$ Then %$p'(s)<0$ and
$$p(\tau)p(\sigma)=(r'-r)\left(1-r\left(\frac{\sigma}{\tau}\right)^n\right)<0,$$
if $n$ is big enough. From \eqref{condi} we easily get the relation
\begin{equation}\label{tast}
\lim_{n\to \infty} n(s_n - \tau)= \tau\log \frac{r'}{r}.
\end{equation}
Further, by a similar analysis as in \cite[Section~9]{ar} we obtain
that $$\mathbb K(z,f_n)=\left\{
                     \begin{array}{ll}
                      \left(\frac{s}{\sqrt{s^2-\tau^2\rho^2(\tau)\varrho^2(s)
}}+\frac{\sqrt{s^2-\tau^2\rho^2(\tau)\varrho^2(s)}}{s}\right)
 & \hbox{if }\ \ s_n\le s\le \sigma,
 \\
                       \frac{1}{2}\left(n+\frac 1n\right)& \hbox{if } \tau\le s\le s_n. \\
                       %1 & \hbox{if } s\le \tau.
                     \end{array}
                   \right.$$ Finally by \eqref{tast}
\[\begin{split}\lim_{n\to \infty} \int_{A'(\tau,\sigma)} \mathbb K(z,
f_n)\rho^2(s) &= \int_{A'(\tau,\sigma)} \mathbb K(z,
f^{\#})\rho^2(s)+\lim_{n\to \infty} \int_{A'(\tau,s_n)} \mathbb K(z,
f_n)\rho^2(s)\\&=\int_{A'} \mathbb K(z, f^{\#})\rho^2(s) +
2\pi\lim_{n\to\infty}\frac n2\int_{\tau}^{s_n}\rho^2(s)s ds
\\&=\int_{A'} \mathbb K(z, f^{\#})\rho^2(s) +
\pi\rho^2(\tau)\lim_{n\to\infty}\frac n2({s_n}^2-\tau^2)
\\&=\int_{A'} \mathbb K(z, f^{\#})\rho^2(s) +
\frac{\tau^2\rho^2(\tau)}{2}\Mod A(r,r').\end{split}\]
This finishes
the proof of Theorem~\ref{te}.
\end{proof}
\section{Proof of Corollaries~\ref{co} and \ref{ro}}
In order to apply Theorems~\ref{or} and \ref{te} to energy
minimization problems for the inverse mappings (see \eqref{po}),
proving in this way Corollaries~\ref{co} and \ref{ro}, we need to
establish the following
\begin{lemma}[The correction lemma]\label{corect} Let $D$ be doubly connected domain in $\Bbb C$ and let $A'=\{z: \tau<|z|<\sigma\}$. Let $h : D \to A'$,  be a homeomorphism of finite Dirichlet
energy, $$E_\rho[h] =\int_{D}\rho^2\circ h (|h_z|^2+|h_{\bar z}|^2)
dm(z)<\infty.$$ Assume that $\rho$ is a radial metric with bounded
Gauss curvature satisfying
\begin{equation}\label{bound}0<\inf_{w\in
A'}\rho^2(w)\le \sup_{w\in A'}\rho^2(w)<\infty.\end{equation}

Then there exists a homeomorphism $\tilde h: D\to A'$ such that
$$ E_\rho[\tilde h ]\le E_\rho[h].$$ The inverse $ \tilde f = \tilde h^{-1}$ belongs to
$W^{1,1}(A',D)$ and has finite distortion. We have the identity
$$\mathcal K_\rho[\tilde f]=E_\rho[\tilde h].$$
\end{lemma}

\begin{proof} We proceed as in \cite[Lemma~7]{ar}, however due to the different intrinsic geometry, we should provide a different proof.
A domain $G$ is said to be convex with respect to the metric $\rho$,
if for $z,w\in G$, there exists a geodesic line $l[z,w]\subset G$,
$z,w\in l[z,w]$ with respect to the metric $\rho$.  By following an
argument of Buser (\cite[Pages~116-121]{bus}, see also
\cite[The~proof~of~Theorem~1]{bus1}) there exists a triangulation of
the Riemann surface $(A', \rho)$ by a locally finite set of
triangles $\Delta_j$ $j = 1, \dots, m$ with pairwise disjoint
interiors, such that the edges of triangles are geodesic arcs or
part of boundary (such domains are lipschitz and convex domains (if
the diameter is small enough)). Also we can assume that
$$\mathrm{diam}_\rho(\Delta_j):=\sup\{d_{\rho}(z,w): z,w\in \Delta_j\}<\varepsilon$$ where $\varepsilon$ is a positive constant
to be chosen later.

Let us sketch the construction of these geodesic triangles. Let
$\mathcal P$ be a finite set of points in $A$ such that any two
points of $\mathcal P$ lie at $\rho$-distance $\le \varepsilon$ and
$\ge \varepsilon/2$ from each other.

Such a set $\mathcal P$ is easily obtained by successively marking
points in $\overline{A'}$ at pairwise distances $\ge \varepsilon/2$
until there is no more room for such points. The set $\mathcal P$ is
finite because the $\rho$-diameter of $A'$ is finite. We add to the
set $\mathcal P$ a maximal finite set of boundary points at pairwise
distance between $\varepsilon/2$ and $\varepsilon$. For arbitrary
triple of points $a,b,c$ from $\mathcal P$ we can construct a
(possible degenerated) triangle $\Delta_i, i=1,\dots, n,$
$n=\binom{|\mathcal P|}{3}$ defined as follows. If $a,b,c$ are
inside $A'$ then the wedges are geodesic. If for example $a,b$
belong to the same connected component of $
\partial A'$, then the wedge $ab$ is the smaller boundary arc bounded by $a$ and $b$.  We exclude all degenerated triangles, those who contain
more than three points inside and those whose diameter is larger
than $\varepsilon$. Thus we get the family $\Delta_j$, $j = 1,
\dots, m$. Then each $\Delta_j$ is contained in a geodesic disk
$B_\varepsilon(p)$ with center at $p\in\Delta_j$ and radius
$\varepsilon$.

Through the homeomorphism $h$ we have a decomposition
$$D= \bigcup_{k=1}^m \overline{D_k},$$ where $\overline{D_k} =
h^{-1}(\overline{\Delta_k})$

Let the Gauss curvature of the metric $\rho$ satisfies $K\le
\kappa^2.$ Let $\mathcal C$ be a positive constant, smaller or equal
to the minimal distance of a given point $p\in A'$ from its
cut-locus. Because the metric is radial, the cut-locus of a point
$p\in A'$ is the arc $[-\tau p,-\sigma p]$. Since the metric
satisfies the condition \eqref{bound} it follows that
$$\mathcal C=\inf\{d_\rho(x,-y): \tau\le x,y\le \sigma\}>0.$$ Take
\begin{equation}\label{eps}\varepsilon<\min\left\{\frac{\pi}{2\kappa},\mathcal C\right\}.\end{equation} Consider the Dirichlet problem of
finding a $\rho-$harmonic map $h_j\colon D_j\rightarrow A'$ with the
given boundary values: $h_j|_{\partial D_j}=h$. Since $h_j(\partial
D_j)$ is contained in a geodesic disk $B_\varepsilon(p)$ with:
\begin{enumerate} \item{} radius $\varepsilon<\pi/(2\kappa)$;
\item{} the cut locus of the centre $p$ disjoint from $B_\varepsilon(p)$,
\end{enumerate} by a result of Hildebrandt, Kaul and Widman \cite{hil} this
Dirichlet problem has a solution contained in $B_\varepsilon(p)$.

Moreover by a result of Jost \cite{jj} we obtain that, since
$h_j\colon\partial D_j\rightarrow A'$ is a homeomorphism onto a
Lipschitz convex curve $\partial\Delta_j$, then the above solution
$h_j$ is a homeomorphism.

Let $$\tilde h=\sum_{j=1}^m h_j(z) \chi_{\overline D_j}(z).$$ Then
$\tilde h$ is a homeomorphism by construction. By using the well
known energy estimates we have
$$\int_{D_j} \|D h_j\|^2 \rho^2(z)dm(z)\le \int_{D_j} \|D h\|^2
\rho^2(z)dm(z).$$ Let $\tilde f={\tilde h}^{-1}$.  Proceeding as in
\cite[Lemma~7]{ar}, we obtain
$$\mathcal K_\rho[\tilde f]\le E_\rho[h].$$

To continue observe that
$$E[h] =\int_{D}(|h_z|^2+|h_{\bar z}|^2)
dm(z)\le \int_{D} \frac{\rho^2\circ h}{\inf_{w\in
A'}\rho^2(w)}(|h_z|^2+|h_{\bar z}|^2) dm(z)<\infty.$$
 By using a recent result of Hencl,
Koskela and Onninen \cite{koskela}, we obtain that the homeomorphism
$f: A' \to D$ of integrable distortion in the Sobolev class
$W^{1,1}(A', D)$ has its inverse $\tilde h\in W^{1,2}(D,A')$. By
introducing the change of variables $w=\tilde h(z)$, proceeding as
in \cite{koskela}, by making use of formulas $$\|B\| = \|B^T \|, \;
\left(A^T\right)^{-1}  \det A = \mathrm{adj}\, A,\; \|\mathrm{adj}\,
A\| = \|A\|
$$ we obtain the identity
$$\mathcal K_\rho[\tilde f]= E_\rho[\tilde h].$$ This completes the proof of the
Lemma~\ref{corect}.

\end{proof}
\begin{remark}
We think that some of the conditions of Lemma~\ref{corect}
concerning the metric $\rho$ are superfluous.
\end{remark}

\section{Appendix} In this section we will give some examples of
\emph{regular} metrics.
\subsection{Hyperbolic metrics}\label{hypo} For every hyperbolic
Riemann surface, the fundamental group is isomorphic to a Fuchsian
group, and thus the surface can be modeled by a Fuchsian model $\Bbb
U/\Gamma$, where $\Bbb U$ is the unit disk and $\Gamma$ is the
Fuchsian group (\cite{las}). If $\Omega$ is a hyperbolic region in
the Riemann sphere $\overline{\Bbb C}$; i.e., $\Omega$ is open and
connected with its complement $\Omega^c := \overline{\Bbb
C}\setminus \Omega$ possessing at least three points. Each such
$\Omega$ carries a unique maximal constant curvature $-1$ conformal
metric $\lambda |dz| = \lambda_\Omega(z)|dz|$ referred to as the
Poincar\'e hyperbolic metric in $\Omega$. The domain monotonicity
property, that larger regions have smaller metrics, is a direct
consequence of Schwarz's Lemma. Except for a short list of special
cases, the actual calculation of any given hyperbolic metric is
notoriously difficult.

By the formula

$$\rho_\Sigma(z) = h(|z|^2),$$
we obtain that the Gauss curvature is given by
$$K = \frac{4(|z|^2{h'}^2 - |z|^2hh''-hh')}{h^4}.$$
Setting $t = |z|^2$, we obtain that
\begin{equation}K =
-\frac{1}{h^2}\left(\frac{4th'(t)}{h}\right)'.\end{equation} As
$K\le 0$ it follows that
$$\left(\frac{4th'(t)}{h}\right)'\ge 0.$$
Therefore the function $$\frac{4th'(t)}{h}$$ is increasing, i.e.
\begin{equation}\label{incre}
t\ge s \Rightarrow \frac{4th'(t)}{h(t)} \ge \frac{4sh'(s)}{h(s)},
\end{equation}
and in particular \begin{equation}\label{incre0} t\ge 0 \Rightarrow
\frac{4th'(t)}{h(t)} \ge 0.
\end{equation}
{\it In this case we obtain that $h$ is an increasing function.}

The examples of hyperbolic surfaces are:

a)  The Poincar\'e disk $\Bbb U$ with the hyperbolic metric
$$\lambda = \frac{2}{1-|z|^2}.$$

b) The punctured hyperbolic unit disk $\Delta = \Bbb U
\setminus\{0\}$. The linear density of the hyperbolic metric on
$\Delta$ is $$\lambda_\Delta = \frac{1}{|z|\log\frac{1}{|z|}}.$$

c) The hyperbolic annulus $A(1/R,R)$, $R>1$. The hyperbolic metric
is given by

$$ h_R(|z|^2)=\lambda_R(z) = \frac{\pi/2}{|z|\log R}\sec\left(
\frac{\pi}{2}\frac{\log|z|}{\log R}\right).$$

In all these cases the Gauss curvature is $K = -1$.

\subsection{Riemann metrics} In the case of the Riemann sphere, the
Gauss-Bonnet theorem implies that a constant-curvature metric must
have positive curvature $K$. It follows that the metric must be
isometric to the sphere of radius
 $1/\sqrt K$ in $\mathbb{R}^3$ via stereographic projection.

d) In the $z$-chart on the Riemann sphere, the metric with $K = 1$
is given by
$$ds^2=h_R^2(|z|^2)|dz|^2 = \dfrac{4|dz|^2}{(1+|z|^2)^2}.$$

e)  Another important case is Hamilton cigar soliton or in physics
is known as {\it Wittens's black hole}. It is a K\"ahler metric
defined on $\Bbb C$.
$$ds^2=h^2(|z|^2)|dz|^2 = \frac{|dz|^2}{1 + |z|^2}.$$ The Gauss curvature is given by
$$K = \frac{2}{1+|z|^2}.$$
In both these cases $K> 0$. This means that
$$\left(\frac{4th'(t)}{h}\right)'\le
0.$$ Therefore the function $$\frac{4th'(t)}{h}$$ is decreasing i.e.
\begin{equation}\label{decre}
t\ge s \Rightarrow \frac{4th'(t)}{h(t)} \le \frac{4sh'(s)}{h(s)}.
\end{equation}
{\it In this case we obtain that $h$ is a decreasing function.}

\end{document}